\documentclass[dvips,ejs,preprint]{imsart}

\RequirePackage[OT1]{fontenc}
\RequirePackage{amsthm,amsmath}
\RequirePackage[numbers]{natbib}
\RequirePackage[colorlinks,citecolor=blue,urlcolor=blue]{hyperref}
\RequirePackage{hypernat}
\usepackage{algorithm}
\usepackage{algorithmic}

\arxiv{math.PR/0000000}

\startlocaldefs
\numberwithin{equation}{section}
\theoremstyle{plain}
\newtheorem{thm}{Theorem}[section]
\endlocaldefs

\begin{document}

\begin{frontmatter}

\title{The Minimum of the Entropy of a Two-Dimensional Distribution with Given Marginals}
\runtitle{Minimum Entropy}

\begin{aug}
\author{\fnms{Giorgio} \snm{Dall'Aglio}\ead[label=e1]{giorgio.dallaglio@fastwebnet.it}}
\and
\author{\fnms{Elisabetta} \snm{Bona}
}

\runauthor{G. Dall'Aglio and E. Bona}

\address{Department of Statistics\\
University La Sapienza\\
p.le Aldo Moro n. 5, 00185, Roma, Italy.\\
\printead{e1}\\
}
\end{aug}

\begin{abstract}
The paper search for the minimum of the entropy of a two-dimensional distribution in the Fr\'{e}chet class, the class of distributions with given marginals. The main result for discrete distributions is an algorithm for building the minimizing distribution, which is given by the maximum distribution function of the Fr\'{e}chet class after a suitable rearrangement of the rows and of the columns. For absolutely continuous distributions a minimum does not exists, and the infimum is equal to $-\infty$.
\end{abstract}


\end{frontmatter}

\section{The problem}

Given an unidimensional discrete random variable  $X$  the entropy $H(X)$ is defined by

\begin{equation*}
    H(X)=-\sum_r p_r \log p_r
\end{equation*}

\noindent where the the $p_r$'s are the probabilities with which $X$ takes its different values $x_r$. They form a countable set, with $p_r > 0$, $\sum_r p_r =1$. We will usually employ for sums and integrals this simplified notation. If  the value $p=0$ will occur, the product $p \log p$ will be taken by continuity to be zero. By its definition, entropy is non-negative; it may also be $+\infty$. The values of the r.v. are not taken in account. The basis of the logarithm is inessential in this paper, since we proceed by comparison; for calculation we take  \emph{e} as basis.

For discrete two-dimensional distributions the definition is:

\begin{equation*}
    H(X,Y)=-\sum_r \sum_s p_{r,s} \log p_{r,s}
\end{equation*}

As usual we pose

\begin{equation*}
    p_{r,\cdot}= P(X=x_r)=\sum_s p_{r,s}
\end{equation*}

\begin{equation*}
    p_{\cdot,s}= P(X=x_s)=\sum_r p_{r,s}
\end{equation*}

We adopt also the notation $H(P)$ were $P =\{p_{r,s}\}$  is the double array of the probabilities.
Entropy is often defined as a function of a random variable, although it depends only, and partially, on its distribution. By its definition, a permutation of the values of the random variables does not change the value of the entropy.

Entropy was defined by Claude E. Shannon \cite{r5} as a measure of uncertainty. This meaning is stressed by the fact that its minimum is zero (no uncertainty) when one of the $p_{r}$'s is equal to $1$, while its maximum, for a given number $n$ of values, is attained when all the probabilities are equal; then the entropy is equal to $\log n$.

In statistics entropy is considered as a measure of uncertainty, more specifically of dispersion. But its main utilization in the theory of information, where it plays a relevant role. In fact entropy changed of sign can indicate the amount of information given by the distribution.

Information theory considers also the mutual information $I(X,Y)$ defined by

\begin{equation*}
    I(X,Y)= \sum_r \sum_s p_{r,s} \log \frac{p_{r,s}}{p_{r,\cdot}p_{\cdot,s}}
\end{equation*}

Simple calculations bring to the relation

\begin{equation*}
	I(X,Y) = H(X) + H(Y) - H(X;Y)
\end{equation*}

\noindent which shows that the search for the minimum of $H(X,Y)$ furnishes also the maximum of $I(X,Y)$. The \emph{maximum mutual information criterion} is largely employed in applications.

For the basic notions about entropy and information we refer to \cite{r6}. Remark that when we write \emph{min}$H$ or \emph{max}$H$ without other indications we will mean always in the Fr\'{e}chet class, that is among the two-dimensional distributions with fixed marginals.

We recall here some notions about distributions with given marginals \cite{r1}. If $F(x)$ and $G(y)$ are two one-dimensional distribution functions, we consider the set $\Gamma(F;G)$ of two-dimensional distributions functions which have $F$ and $G$ as marginals; $\Gamma(F;G)=\{F(x,y): F(x,y)  \text{ is a distribution function,}\\ \lim_{y \rightarrow \infty} F(x,y) = F(x); \lim_{x \rightarrow \infty} F(x,y) = G(y)\}$.

This is called also \emph{Fr\'{e}chet class}, after a paper by Maurice Frech\'{e}t \cite{r2} which started the study of the subject, after a relevant but almost ignored paper by Wassilly Hoeffding  \cite{r3}.

Following inequalities hold.

\small{
\begin{equation}
	\max\{ F(x) + F(y)-1, 0\}=W(x,y) \leq F(x,y) \leq M(x,y)=\min\{F(x),G(y)\}
\label{1}
\end{equation}
}

The functions on the right and on the left side belong to the class, being so in it the maximum and the minimum one. $M(x,y)$ furnishes the maximum concentration, with a functional relation between $X$ and $Y$: the distribution is concentrated on the set $(F(X) = G(Y))$ , so that the relations $Y=G^{-1}F(X)$ and $X=F^{-1}G(Y)$ hold almost surely. Similar results hold for $W$.

In the discrete case, to which we are interested now, the probabilities $p_{r,s}$ are subject to the inequalities

\begin{equation}
	\max\{p_{r,\cdot} + p_{\cdot,s} -1, 0\} \leq p_{r,s} \leq \min\{p_{r,\cdot}, p_{\cdot,s}\}
\label{2}
\end{equation}

Either bound can be reached for every couple $(r,s)$, but not always at the same time for more than one couple. The distribution functions  $M$  and  $W$ correspond respectively to the \emph{cograduation table} and the \emph{contrograduation table}, introduced by T. Salvemini \cite{r4}. The cograduation table is constructed giving to $p_{1,1}$ its maximum value; then all the other items in the same row or in the same column, or both, are null, and the construction continues recursively. This construction is known also, specially in Operations Research, as  \emph{NW rule}.

\section{The dichotomic case}

We start with the simple case of a dichotomic table, i.e. both $X$ and $Y$ assume only two values. Then

\begin{equation*}
    H(P)= - p_{1,1} \log p_{1,1} - p_{1,2} \log p_{1,2} - p_{2,1} \log p_{2,1} - p_{2,2} \log p_{2,2}
\end{equation*}

Expressing the probabilities as functions of the only variable $p_{1,1}$, we have

\begin{eqnarray*}
  H(P) &=& H(p_{1,1}) = - p_{1,1} \log p_{1,1} - (p_{1,\cdot} - p_{1,1}) \log (p_{1,\cdot} - p_{1,1}) \\
       &-& (p_{\cdot,1} - p_{1,1}) \log (p_{\cdot,1} - p_{1,1}) - (1 - p_{1,\cdot} - p_{\cdot,1} + p_{1,1}) \log (1 - p_{1,\cdot} - p_{\cdot,1} + p_{1,1})
\end{eqnarray*}

The derivative of $H(p_{1,1})$ is

\begin{eqnarray*}
  H'(p_{1,1}) &=& - \log p_{1,1} - 1 + \log (p_{1,\cdot} - p_{1,1}) + 1 + \log (p_{\cdot,1} - p_{1,1}) + 1 - \log (1 - p_{1,\cdot} - p_{\cdot,1} + p_{1,1}) - 1 \\
  &=& \log[(p_{1,\cdot} - p_{1,1}) (p_{\cdot,1} - p_{1,1})]- \log[p_{1,1} (1 - p_{1,\cdot} - p_{\cdot,1} + p_{1,1})]
\end{eqnarray*}

\noindent so that

\begin{equation*}
  H'(p_{1,1}) > 0 \Leftrightarrow (p_{1,\cdot} - p_{1,1}) (p_{\cdot,1} - p_{1,1}) - p_{1,1} (1 - p_{1,\cdot} - p_{\cdot,1} + p_{1,1})>0 \Leftrightarrow p_{1,1} < p_{1,\cdot} p_{\cdot,1}
\end{equation*}

It can be easily seen that $p_{1,1}=p_{1,\cdot} p_{\cdot,1}$ satisfies inequalities (\ref{2}). So the maximum of $H(P)$ is given by $p_{1,1}=p_{1,\cdot} p_{\cdot,1}$, i.e. by the independence table.

We must search now for the minimum. Suppose

\begin{equation}
  p_{1,\cdot} \geq p_{2,\cdot}, \qquad p_{\cdot,1} \geq p_{\cdot,2},\qquad p_{1,\cdot} \leq p_{\cdot,1}
\label{3}
\end{equation}

Then (\ref{2}) becomes

\begin{equation*}
     p_{1,\cdot} + p_{\cdot,1} - 1 \leq p_{1,1} \leq p_{1,\cdot}
\end{equation*}

\noindent and the minimum obtains for one of the extreme values. Denote respectively $P'$ and $P''$ the tables obtained putting $p_{1,1}=p_{1,\cdot}$ and $p_{1,1}=p_{1,\cdot}+p_{\cdot,1}-1$. We may write

\begin{eqnarray*}
  H(P')-H(P'') &=& - p_{1,\cdot} \log p_{1,\cdot} - (p_{\cdot,1}-p_{1,\cdot}) \log (p_{\cdot,1}-p_{1,\cdot}) \\
  &+& (p_{1,\cdot}+p_{\cdot,1}-1) \log (p_{1,\cdot}+p_{\cdot,1}-1) + (1-p_{1,\cdot})\log(1-p_{1,\cdot})
\end{eqnarray*}

\noindent and

\begin{eqnarray*}
  \frac{d}{d p_{1,\cdot}}[H(P')-H(P'')] &=& - \log p_{1,\cdot} - 1 + \log(p_{\cdot,1}-p_{1,\cdot}) + 1 + \log(p_{1,\cdot}+p_{\cdot,1} - 1)\\
  &+& 1 - \log(1-p_{1,\cdot}) - 1 \\
  &=& \log [(p_{\cdot,1}-p_{1,\cdot})(p_{1,\cdot}+p_{\cdot,1}-1)] - \log[p_{1,\cdot}(1-p_{1,\cdot})] \leq 0
\end{eqnarray*}

\noindent since

\begin{equation*}
   (p_{1,\cdot}+p_{\cdot,1}-1)(p_{\cdot,1}-p_{1,\cdot})-(1-p_{1,\cdot})p_{1,\cdot} = p_{\cdot,1}^{2} - p_{\cdot,1} \leq 0
\end{equation*}

Then the maximum of  $H(P') - H(P'')$ is given by the minimum of $p_{1,\cdot}$ which, according to (\ref{3}), is $\frac{1}{2}$, and

\begin{eqnarray*}
  H(P')-H(P'') &\leq& - \frac{1}{2} \log \frac{1}{2} - \left(p_{\cdot,1} - \frac{1}{2}\right) \log \left(p_{\cdot,1} - \frac{1}{2}\right)+ \left(p_{\cdot,1} - \frac{1}{2}\right) \log \left(p_{\cdot,1} - \frac{1}{2} \right) \\
  &+& \frac{1}{2} \log \frac{1}{2}= 0
\end{eqnarray*}

We have thus proved.

\begin{thm}
For a $2\times2$ distribution with given marginals the maximum of the entropy is given by the independence table. The minimum obtains choosing the maximum $p_{u,\cdot}$ among the $p_{r,\cdot}$'s and the maximum $p_{\cdot,v}$ among the $p_{\cdot,s}$'s and putting $p_{u,v}=\min\{p_{u,\cdot} p_{\cdot,v} \}$.
\end{thm}

The fact that the maximum of the entropy for discrete distributions is given by the independence distribution is well known, as the resulting inequality

\begin{equation}
  	 H(X,Y) \leq  H(X) + H(Y)
  \label{4}
\end{equation}

\noindent and the intuitive meaning is immediate if we talk of information, since clearly the independence give the minimum information among all two-way distributions, and the minimum information cannot be lesser than the information already contained in the marginals.

As for the minimum, for discrete distributions the following inequality hold

\begin{equation}
  	H(X,Y) \geq  H(X),H(Y)
\label{5}
\end{equation}

\noindent proved by

\begin{eqnarray*}
  H(X,Y) - H(X) &=& \sum_{r} p_{r, \cdot} \log p_{r, \cdot} - \sum_{r,s} p_{r,s} \log p_{r,s} \\
  &=& \sum_{r} \left[ p_{r, \cdot} \log p_{r, \cdot} - \sum_{s} p_{r,s} \log p_{r,s} \right] \\
  &=& \sum_{r} \left[ \log p_{r, \cdot}^{p_{r,.}} - \log \prod_{s} p_{r,s}^{p_{r,s}} \right] \\
  &=& \log \prod_{r,s} \left(\frac{p_{r, \cdot}}{p_{r,s}}\right)^{p_{r,s}} \geq 0
\end{eqnarray*}

\section{Discrete distributions}

We consider now the case where $X$ and $Y$ assume a finite number of values, say $\{x_1,x_2,\ldots,x_m\}$  and $\{y_1,y_2,\ldots,y_n\}$.
Some calculations along the same lines of the dichotomic case show that the maximum of the entropy is given by the independence distribution; a result well known, as already said.
The search for the minimum by derivation is more complicate. Because of the form of the derivative, the maximum for each $p_{r,s}$ is reached in one of the extreme of (\ref{2}), but the comparison is difficult because it depends on the other values. We present an algorithm which gives a minimizing table.

The hint is given by the $2\times2$ case; the algorithm consists of the repeated use of the same step.

\begin{algorithm}
\label{algorithm}
\caption{}
\begin{itemize}
  \item [] Choose the maximum (or one of the maximums) among the $p_{r, \cdot}$'s, say $p_{u, \cdot}$
  \item [] Choose the maximum (or one of the maximums) among the $p_{\cdot, s}$'s , say  $p_{\cdot, v}$
  \item [] Put $p^{*}_{u, v}=\min\{p_{u, \cdot},p_{\cdot, v}\} $
  \item [] Delete the row, or the column, or both, in which there is only one entry different from\\
  \quad zero and continue in the same way for subsequent tables.
\end{itemize}
\end{algorithm}

We remark that the table resulting after deleting is not a correlation table since the entries do not sum up to $1$. We could obtain a correlation table by dividing all the entries by $1-p_{u, \cdot}$ or $1-p_{\cdot, v}$; but since we proceed by comparison it is easy to see that our construction arrives to the same result.

If we rearrange the rows and the columns in the order in which we have taken the minimizing values $p^{*}_{r, s}$, $P^{*}$ is the cograduation table, built according the \emph{NW corner rule}. It corresponds to the maximum distribution function in the Fr\'{e}chet class, i.e the function $M$ in (\ref{1}).

\begin{thm}
Given two r.v.' $X$ and $Y$ with fixed discrete distributions, and $H(X)$, $H(Y)<+\infty$, the minimum of the entropy $H(X,Y) = H(P)$ is given by the correlation table $P^{*}$ built with the Algorithm $1$.
\end{thm}

\begin{proof}
We start with the finite case, and we proceed by induction on the number of rows or columns; the theorem is trivially true when the number of rows, or of columns, is equal to $1$.

Let us apply the algorithm, arriving to the correlation table $P^{*}$. For sake of simplicity suppose that the first step brings to $p^{*}_{1, 1}=p_{1,\cdot}$, so that $p_{1,s}=0$ for $s>1$.
In the table $P^{*}$ cancel the first column and call $\overline{P}^{*}$  the remaining table. Also for a generic $P$ call $\overline{P}$  the table obtained canceling the first column. We have

\begin{eqnarray}
  \nonumber H(P^{*}) &=& -p_{1,\cdot} \log p_{1,\cdot} + H(\overline{P}^{*}) \\
  H(P) &=& - \sum_{s} p_{1,s} \log p_{1,s} + H(\overline{P})
  \label{6}
\end{eqnarray}

Now

\begin{equation*}
    p_{1, \cdot} \log p_{1, \cdot} - \sum_{s} p_{1,s} \log p_{1,s} = \log \frac{p_{1, \cdot}^{p_{1, \cdot}}}{\prod_{s} p_{1, s}^{p_{1, s}}} = \log \prod_{s} \left( \frac{p_{1, \cdot}}{p_{1,s}} \right)^{p_{1,s}} \geq 0
\end{equation*}

\noindent so that

\begin{equation}
    -p_{1, \cdot} \log p_{1, \cdot} \leq \sum_{s} -p_{1,s} \log p_{1,s}
\label{7}
\end{equation}

Moreover, since $\overline{P}^{*}$  is built according to the Algorithm $1$, by the recursive construction $H(\overline{P}^{*}) \leq H(\overline{P})$. This, along with (\ref{6}) and (\ref{7}), gives $H(P^{*}) \leq H(P)$, proving the theorem for the finite case.

In the denumerable case, consider the set $D_n=\{(r,s):r,s \geq n\}$ and its complement $D_{n}^{c}$  and write

\begin{eqnarray}
\nonumber  H(X,Y) &=& -\sum_{r,s} p_{r,s} \log p_{r,s} \\
    &=& -\sum_{(r,s) \in D_n} p_{r,s} \log p_{r,s} - \sum_{(r,s) \in D_{n}^{c}} p_{r,s} \log p_{r,s}
\label{8}
\end{eqnarray}

From the hypothesis $H(X),H(Y)\leq +\infty$ and (\ref{4}) it follows that $H(X,Y)\leq +\infty$, and this implies that the last sum in (\ref{8}) tends to zero when $n\rightarrow\infty$. Therefore applying the Algorithm $1$ to the sum restricted to $D_n$  and going to the limit proves the theorem.

\end{proof}

\section{Continuous distributions}

Entropy for continuous distributions is defined by

\begin{eqnarray}
\nonumber  H(X) &=& -\int_R f(x) \log f(x) dx\\
           H(X,Y) &=& -\int_{R_2} f(x,y) \log f(x,y) dx dy
\label{9}
\end{eqnarray}

The second definition in (\ref{9}) contains the two-dimensional density function $f(x,y)$, whose existence is not assured by the absolute continuity of the marginals, contrary to what happens for discrete distributions.
	This remark must be kept in mind if one proceeds by limit, as we will do: the limit of a sequence of absolutely continuous distributions is not necessarily absolutely continuous, even maintaining the same marginals; it can be, for instance, concentrated on a line, so that it has not a two-dimensional density function.

Entropy for continuous distributions presents features very different from the discrete case. The differences arise from the fact that the density function may assume values greater than $1$, so that the entropy may be negative.

We will proceed by discretization and limit. Given a two-dimensional r.v. $(X,Y)$ with density function $f(x,y)$, which we will assume continuous, define

\begin{eqnarray}
\nonumber  X_n &=& \frac{r}{n} \qquad \text{if} \qquad \frac{r}{n} \leq X < \frac{r+1}{n} \qquad r=0,\pm 1,\pm 2, \ldots\\
           Y_n &=& \frac{s}{n} \qquad \text{if} \qquad  \frac{s}{n} \leq X < \frac{s+1}{n} \qquad s=0,\pm 1,\pm 2, \ldots
\label{10}
\end{eqnarray}

Then $X_n$ (resp. $Y_n$) converges almost surely to $X$ (resp $Y$) and $(X_n,Y_n)$ converges almost surely to $(X,Y)$. For the probabilities, after defining

\begin{equation*}
          \Delta_{n,r,s} = \left\{ \frac{r}{n} \leq x < \frac{r+1}{n}, \frac{s}{n} \leq y < \frac{s+1}{n}\right\}
\end{equation*}

\noindent we have

%

\begin{eqnarray}
\nonumber p_{n,r,s} &=& P\left(X_n =\frac{r}{n},Y_n = \frac{s+1}{n}\right) = \int_{\Delta_{n,r,s}} f(x,y) dx dy\\
 p_{n,r,\cdot}&=&\int_{\frac{r}{n}}^{\frac{r+1}{n}} f(x) dx
\label{11}
\end{eqnarray}

This makes clear the aforesaid difference between the discrete and the continuous case.
The consequences of these remarks are illustrated by the following Theorem (see Theorem 1,3,1 of \cite{r6}). As we have said, now the entropy may be negative, so we require its boundness.

\begin{thm}
Given a r.v. $(X.Y)$ with continuous density function $f(x,y)$ and such that $\mid H(X,Y) \mid < +\infty $, consider the sequence of r.v.'s $\{X_n,Y_n\}$  defined by (\ref{10})
Then $\lim[H(X_n,Y_n) - 2 \log n]=H(X,Y)$
\end{thm}

\begin{proof}
We can write

\begin{eqnarray}
\nonumber H(X_n,Y_n) - 2 \log n &=&  -\sum_{r,s} p_{n,r,s} \log p_{n,r,s} - 2 \log n \\
          &=& -\sum_{r,s} p_{n,r,s} \log (n^2 p_{n,r,s})\\
\nonumber &=& - \sum_{r,s} \int_{\Delta_{n,r,s}} f(x,y) \log \left( n^2 \int_{\Delta_{n,r,s}} f(u,v) du dv \right) dx dy
\label{12}
\end{eqnarray}

Therefore we shall prove that

\begin{eqnarray}
\nonumber  && \lim_{n\rightarrow +\infty}  \left( -\sum_{r,s} \int_{\Delta_{n,r,s}} f(x,y) \log \left(n^2 \int_{\Delta_{n,r,s}} f(u,v) du dv \right) dx dy \right) \\
&& = -\int f(x,y) \log f(x,y) dx dy
\label{13}
\end{eqnarray}

Since $f(x,y)$ is continuous, for $(x,y) \in \Delta_{n,r,s}$  we have

\begin{equation*}
            \lim_{n\rightarrow +\infty} \left( n^2 \int_{\Delta_{n,r,s}} f(u,v) du dv \right) = f(x,y)
\end{equation*}

\noindent and

\begin{equation*}
            \lim_{n\rightarrow +\infty} f(x,y) \log \left( n^2 \int_{\Delta_{n,r,s}} f(u,v) du dv \right) = f(x,y) \log f(x,y)
\end{equation*}

This is not sufficient for the convergence of the integral in (\ref{13}). For proving it, we choose a positive integer m, and consider the set

\begin{equation*}
            D_{m}=\left\{(x,y): \frac{1}{m} \leq f(x,y) \leq m \right\}
\end{equation*}

Since $H(X,Y)$ is finite, we have

\begin{equation}
            \lim_{m \rightarrow \infty} \int_{D_m} -f(x,y) \log f(x,y) dx dy= H(X,Y)
\label{14}
\end{equation}

For $(u,v)$ belonging to $D_m$ it is

\begin{equation*}
            \frac{1}{m} \leq n^2 \int_{\Delta_{n,r,s}} f(u,v) dudv \leq m
\end{equation*}

Then the integrand in (\ref{13}) is bounded, and, by the dominated convergence theorem, for $m$ fixed

\begin{equation*}
         \lim_{n \rightarrow \infty} \left(-\sum_{r,s} \int_{\Delta_{n,r,s} \bigcap D_m} f(x,y) \log \left( n^2 \int_{\Delta_{n,r,s}} f(u,v) dudv \right) dx dy \right) = - \int_{D_m} f(x,y) \log f(x,y) dxdy
\end{equation*}

This means that $\forall \varepsilon>0,m,\exists n_m :$

\small{
\begin{equation*}
      \left|-\sum_{r,s} \int_{\Delta_{n,r,s} \bigcap D_m} f(x,y) \log \left( n^2 \int_{\Delta_{n,r,s}} f(u,v) dudv\right) dx dy - \int_{D_m} f(x,y) \log f(x,y) dx dy \right| \leq \varepsilon
\end{equation*}
}

\noindent i.e. there is a subsequence $\{n_m\}$ for which the limit (\ref{13}) holds. But this can be said also if we start from a subsequence, and this prove the theorem.

\end{proof}

We proceed now to the search for the minimum of $H(X,Y)$ when the marginal distributions $F(x)$ and $G(y)$ are given. We suppose that $F$ and $G$ are strictly increasing: if there are intervals $(x_r,x_r+d_r)$ in which the density is zero, they are a countable set and are irrelevant for the entropy; they can be easily removed passing from F to a new distribution function defined by recurrence:

\begin{equation*}
    F_0(x)=F(x)
\end{equation*}

\begin{equation*}
    F_{i+1}(x)=\left\{
                 \begin{array}{ll}
                   F_i(x) & \hbox{if} \quad x \leq x_i; \\
                   F_i(x+d_r) & \hbox{if} \quad  x>x_i.
                 \end{array}
               \right.
\end{equation*}

We proceed  by discretization. According to Theorem $3.1$, we have

\begin{equation*}
    \min H(X_n, Y_n) = H(X_n^{*}, Y_n^{*})
\end{equation*}

\noindent where $(X_n^{*}, Y_n^{*})$  is built with the Algorithm 1. And because of Theorem $4.1$

\begin{equation}
    \min H(X_n, Y_n) = \lim_{n\rightarrow\infty} [ H(X_n^{*}, Y_n^{*}) - 2 \log n]
\label{15}
\end{equation}

But the limit in (\ref{15}) cannot be obtained through Theorem 3. As remarked for the Algorithm 1, the r.v. $(X_n^{*}, Y_n^{*})$ has a distribution function which is maximum in the Fr\'{e}chet class, and the same is for the limit $(X^{*}, Y^{*})$ when $n\rightarrow\infty$ . Therefore, as said then, $(X^{*}, Y^{*})$  is concentrated in the curve $\Gamma= (F(X) = G(Y))$, where $F$ and $G$ are the marginal distribution functions, and it has not a two-dimensional density, so that Theorem $4.1$ cannot be applied.
Since by hypothesis $G$ is strictly increasing and has continuous derivative, the inverse $G^{-1}$ exists, is strictly increasing, and has continuous derivative; the curve $\Gamma$ can be written as

\begin{equation*}
    Y=G^{-1}(F(X))=t(X)
\end{equation*}

\noindent where the function $t$ has continuous derivative.

The search for the minimum brings to a somewhat surprising result.

\begin{thm}
Let $F(x)$ and $G(y)$ be two one-dimensional distribution functions, with continuous derivatives, and such that  $H(X)$, $H(Y)<+\infty$. Then the infimum of $H(X,Y)$ in the Fr\'{e}chet class of the distribution functions with marginals $F$ and $G$ is $-\infty$. More precisely there is a positive finite $c$ such that

\begin{equation}
    \min H(X, Y) + \log n \rightarrow c
\label{16}
\end{equation}

A minimum in the class does not exists, i.e. there is no distribution function with two-dimensional density which attains the minimum.
\end{thm}

\begin{proof}
The argument will be similar to the one in Theorem $4.1$. For that we introduce the function

\begin{equation*}
    f_{n}^{*}(x,y) = n^2 p^{*}_{n,r,s} \quad \text{if} \quad (x,y) \in \Delta_{n,r,s}
\end{equation*}

\noindent so that

\begin{eqnarray}
           H(X_{n}^{*},Y_{n}^{*}) - 2 \log n &=& \sum_{r,s} p^{*}_{n,r,s} \log n^2 p^{*}_{n,r,s}\\
\nonumber  &=& \sum_{r,s} \int_{\Delta_{n,r,s}} f^{*}_{n}(x,y) \log \left[n^2 \int_{\Delta_{n,r,s}} f^{*}_{n}(u,v) du dv \right]dx dy
\label{17}
\end{eqnarray}

We prove first that out of the curve $\Gamma$ the integral tend to $0$:

\begin{equation}
    \lim \sum_{r,s} \int_{\Delta_{n,r,s} \bigcap \Gamma^c} f^{*}_{n}(x,y) \log \left[n^2 \int_{\Delta_{n,r,s}\bigcap \Gamma^c} f^{*}_{n}(u,v) du dv \right]dxdy=0
\label{18}
\end{equation}

For $(x,y) \in \Gamma^c$, $f^{*}_{n}(x,y)\rightarrow0$  and $f^{*}_{n}(x,y) \log \left(n^2 \int_{\Delta_{n,r,s}\bigcap \Gamma^c} f^{*}_{n}(u,v) du dv \right) \approx f(x,y)\rightarrow0$

The argument trough the sets $D_m$ as in Theorem $4.1$ proves (\ref{18}).

We  go on with the proof, studying the structure of the distribution of the minimizing r.v. $(X_{n}^{*},Y_{n}^{*})$ .
We may suppose that the rows and the columns are in the order in which they are chosen in the construction. This implies that cases with probability different from zero in the same row or in the same columns are adjacent. The first choice is now $p^{*}_{n,1,1}$; suppose that it is equal to $p_{\cdot,1}$; there may be other probabilities  different from zero in the same column, so that $p^{*}_{n,1,s}=p_{\cdot,s}$  for $s=1,2,\ldots,k$ $(k\geq1)$; it is  $k<\infty$ since $p^{*}_{n,1,s}=p_{\cdot,s}$  for all $s$ would imply $p^{*}_{n,r,s}=0$  for $r\neq1$  and every $s$, not consistent with the margins.	After that another string will start. It may be vertical, as the first one, or horizontal, i.e. with  $p^{*}_{n,r,s}=p_{r,\cdot}$. We have therefore a set $A_n$ of indexes $r$ for which the strings are vertical, and for each $s$ two values $s'_{n,r}$  and $s''_{n,r}$  such that $p^{*}_{n,r,s}=p_{\cdot,s}$ , and a set $B_n$ of indexes $s$ for which the strings are horizontal, and for each $s$ two values $r'_{n,s}$  and $r''_{n,s}$   such that $p^{*}_{n,r,s}=p_{r,\cdot}$  . One of the sets  $A_n$  and $B_n$   may be void.

Suppose first that

\begin{equation}
    \sum_{s'_{n,r}\leq s\leq s''_{n,r}} p_{\cdot,s}=p_{r,\cdot} \qquad \text{and} \qquad \sum_{r'_{n,s}\leq r\leq r''_{n,s}} p_{r,\cdot}=p_{\cdot,s}.
\label{19}
\end{equation}

This implies that the strings are disjoint, and
	
\begin{eqnarray}
\nonumber H(X_{n}^{*},Y_{n}^{*})-\log n &=& -\sum_{r,s} p^{*}_{n,r,s} \log n p^{*}_{n,r,s}\\
\nonumber &=& -\sum_{r \in A_{n}} \sum_{s'_{n,r}\leq s\leq s''_{n,r}} p_{\cdot,s} \log (n p_{\cdot,s}) - \sum_{s \in B_{n}} \sum_{r'_{n,s}\leq r\leq r''_{n,s}} p_{r,\cdot}\log (n p_{r,\cdot})\\
\nonumber &=& -\sum_{r \in A_{n}} \sum_{s'_{n,r}\leq s\leq s''_{n,r}} \int_{\frac{s}{n}}^{\frac{s+1}{n}} g(y) \log \left(n \int_{\frac{s}{n}}^{\frac{s+1}{n}} g(v) dv \right) dy  \\
\nonumber &-& \sum_{s \in B_{n}} \sum_{r'_{n,s}\leq r\leq r''_{n,s}} \int_{\frac{r}{n}}^{\frac{r+1}{n}} f(x) \log \left(n \int_{\frac{r}{n}}^{\frac{r+1}{n}} f(u) du \right) dx \\
\nonumber &\approx&   -\sum_{r \in A_{n}} \sum_{s'_{n,r}\leq s\leq s''_{n,r}} \int_{\frac{s}{n}}^{\frac{s+1}{n}} g(y) \log g(y) dy\\
\nonumber &-& \sum_{s \in B_{n}} \sum_{r'_{n,s}\leq r\leq r''_{n,s}} \int_{\frac{r}{n}}^{\frac{r+1}{n}} f(x) \log f(x) dx\\
\nonumber &=&  -\sum_{r \in A_{n}} \int_{\frac{s'_{n,r}}{n}}^{\frac{s''_{n,r}+1}{n}} g(y) \log g(y) dy -\sum_{s \in B_{n}} \int_{\frac{r'_{n,s}}{n}}^{\frac{r''_{n,s}+1}{n}} f(x) \log f(x) dx\\
  &=& -\int_{A'_{n}} g(y) \log g(y) dy - \int_{B'_{n}} f(x) \log f(x) dx
\label{20}
\end{eqnarray}

\noindent where $A'_{n} = \bigcup_{r \in A_n} \left(\frac{s'_{n,r}}{n}, \frac{s''_{n,r}+1}{n}\right)$ and $B'_{n}=\bigcup_{r \in B_n} \left(\frac{r'_{n,s}}{n}, \frac{r''_{n,s}+1}{n}\right)$.

If we take a subsequence of $\{A'_{n}\}$ converging to $A$ and a subsequence of $\{B'_{n}\}$  converging to $B$ , since $H(X)$ and $H(Y)$ are finite the last term of (\ref{21}) converge to

\begin{equation*}
    -\int_{A} g(y) \log g(y) dy - \int_{B} f(x) \log f(x) dx
\end{equation*}

\noindent and this proves the theorem if (\ref{19}) holds. Remark that (\ref{19}) imply that $P(A)+P(B)=1$.

Suppose now that one of the equalities (\ref{19}) do not hold, for instance \\
 $\sum_{s'_{n,1}\leq s\leq s''_{n,1}} p_{\cdot,s} < p_{r,\cdot}$.

This mean that $p^{*}_{n,s'_{n,r}-1}$  or $p^{*}_{n,s''_{n,r}+1}$   or both are greater than zero. But

\begin{equation*}
    \Delta_{n,r,s} \bigcap \Gamma = \left[ \left(\frac{r}{n} \leq x < \frac{r+1}{n} \right), \left(t \left(\frac{r}{n} \leq y < \frac{s'_{n,r}}{n} \right)\right) \right]
\end{equation*}

\noindent and

\begin{equation*}
    P(\Delta_{n,r,s} \bigcap \Gamma) = P \left(t \left(\frac{r}{n} \right) \leq y < \frac{s'_{n,r}}{n} \right)
\end{equation*}

This holds also for the other cases in the strings, both in $A_n$  and in $B_n$. Then all the strings are disjoint, and we are in the same situation as when (\ref{19}) hold. And this concludes the proof.

\end{proof}

\section{Examples}

The complicate way to arrive to the minimum allows simple result only for special cases.

\vspace{0.5cm}
\noindent \textbf{Example 1.} \emph{If  X and Y have the same discrete distribution, then $\min H(X,Y) = H(X) = H(Y)$.} \\
\noindent Take the maximum, or one of the maximums, among the probabilities of each margin, say $p_{1,\cdot}$,$p_{\cdot,1}$. Then $p_{1,1}^{*}= p_{1,\cdot}$ or $p_{\cdot,1}$. This cancels one row and one column, while the remaining marginal probabilities remain unchanged, so that we may continue in the same way. As a result, the minimizing distribution has $p_{r,r}^{*}=p_{r,\cdot}$ for any $r$, $p_{r,s}^{*}=0$  for $r \neq s$  and

\begin{equation*}
    H(X^{*},Y^{*})=-\sum_{r} p_{r,r}^{*} \log p_{r,r}^{*} = - \sum_{r}p_{r,\cdot} \log p_{r,\cdot} =H(X^{*})=H(X)
\end{equation*}

This expression shows two distributions, the first one two-dimensional, the other one unidimensional but the two entropies have the same value.

\vspace{0.5cm}
\noindent \textbf{Example 2.} \emph{If $X$ and $Y$ have uniform discrete distributions, of length respectively  $m$  and  $kn$ , with $k$ integer, then $\min H(X,Y) = H(X)$. } \\
\noindent We proceed by induction: if $k = 1$, it is the case of Example 1, which now gives $H(X) = \log n$. If $k > 1$ , we have $p_{r,\cdot} = \frac{1}{m} = \frac{1}{kn}$ and $p_{\cdot,s} = \frac{1}{n} > \frac{1}{m}$, therefore we can make $p_{1,1}^{*} = \frac{1}{kn}$. Now the marginal probabilities are unchanged, except  $p_{\cdot,1}$ , which is equal to $\frac{1}{n} - \frac{1}{kn}$; we can make $p_{2,2}=\frac{1}{kn}$, and continue until $p_{n,n}^{*} = \frac{1}{kn}$. The result is a square table of dimension $n \times n$, plus a table of dimensions $\frac{k-1}{n} \times n$. This gives the result. \\
 If  $m = kn + r$ , with $r$ different from zero, we may proceed as above for $kn$ steps, and it remains a $r \times n$  table, with which we may proceed again in the same way.

\vspace{0.5cm}
\noindent \textbf{Example 3.} We have found some cases in which  $H(X,Y) = H(X)$. Let us investigate when this happens. \emph{Suppose that $X$ and $Y$ have discrete distributions, of length respectively  $m$  and  $n$. Then  $H(X,Y) = H(X)$ if there is a partition $\{I_1, I_2,\ldots.\}$ such that}

\begin{equation}
    \sum_{r \in I_s} p_{r,\cdot}= p_{\cdot,s} \qquad s=1,2,\ldots
\label{21}
\end{equation}

\noindent If (\ref{21}) is satisfied, we may put

\begin{equation*}
    p_{r,s}^{*}= p_{r,\cdot} \quad \text{for} \quad r \in I_s
\end{equation*}

\noindent so that the distribution $\{p_{r,s}^{*}\}$ is consistent with the marginals, and $H(X^{*},Y^{*}) = H(X)$ gives the minimum by (\ref{5}). On the other hand, if  $H(X,Y) = H(X)$, supposing, as already made, that the probabilities $p_{\cdot,s}$ and $p_{r,\cdot}$ are decreasing, it cannot be $p_{1,\cdot} > p_{\cdot,1}$ , since all the entries $p_{r,s}^{*}$  would be lesser than $p_{1,\cdot}$  and $p_{1,\cdot}$  could not appear among the entries $p_{r,s}^{*}$; continuing so we prove the assertion.
\noindent We remark that, in agreement with (\ref{5}), $H(X) \geq  H(Y)$. In fact

\begin{eqnarray*}
            H(X)-H(Y) &=& \sum_{s} p_{\cdot,s} \log p_{\cdot,s} - \sum_{r} p_{r,\cdot} \log p_{r,\cdot}\\
\nonumber   &=& \sum_{s} \left[ p_{\cdot,s} \log p_{\cdot,s} - \sum_{r \in I_s} p_{r, \cdot} p_{r, \cdot} \right]\\
\nonumber   &=& \sum_{s} \log \frac{p_{\cdot,s}^{p_{\cdot,s}}}{\prod p_{r,\cdot}^{p_{r,\cdot}}}= \sum_{s} \prod_{r} \log \left(\frac{p_{\cdot,s}}{p_{r,\cdot}}\right)^{p_{r,\cdot}} \geq 0
\end{eqnarray*}

\noindent with equality holding iff  $X$  and  $Y$  have the same distribution.

A particular case is when $X$ assumes two values and $Y$ has geometric distribution, more exactly

\begin{equation*}
    p_{1, \cdot}=p; \quad p_{2,\cdot}=q=1-p
\end{equation*}

\begin{equation*}
    p_{\cdot,s}=q^{s-1}p; \quad s=1,2,\ldots \\
\end{equation*}

Now  $p_{\cdot,1}=p_{1, \cdot}$  and $p_{\cdot,2}=\sum_{r>1}p_{r,\cdot}$  so that by Example $3$ \emph{min}$H(X,Y) = H(X) \\
= - \log p - \frac{q}{p} \log p$. Also the direct calculation is very easy.

\vspace{0.5cm}
\noindent \textbf{Example 4.} Consider now a case in which $X$ and $Y$ have respectively geometric and uniform distributions; more precisely

\begin{eqnarray*}
    p_{r,\cdot}&=&\frac{k-1}{k} \left(\frac{1}{k}\right)^{r-1}, \quad r=1,2,\ldots\\
    p_{\cdot,s}&=&\frac{1}{k}, \quad s=1,2,\ldots,k
\end{eqnarray*}

Then

\begin{equation*}
    p_{1,1}^{*}= \min \left\{\frac{k-1}{k}, \frac{1}{k}\right\}=\frac{1}{k}
\end{equation*}

\noindent and, recursively,

\begin{equation*}
    p_{1,s}^{*}= \min \left\{\frac{k-1}{k}-\frac{s-1}{k}, \frac{1}{k}\right\}=\frac{1}{k} \quad \text{for} \quad s=2,3, \ldots,k-1
\end{equation*}

\noindent The result is that the entries of $P^{*}$ are null, except

\begin{eqnarray*}
    p_{1,s}^{*}&=&p_{\cdot,s}  \quad \text{for} \quad s=1,2,3, \ldots,k-1 \\
    p_{r,\cdot}&=&p_{r,\cdot}  \quad \text{for} \quad r=2,3, \ldots
\end{eqnarray*}

Then

\begin{eqnarray*}
    H(X^{*},Y^{*}) &=& -\sum_{1 \leq s \leq k-1} p_{\cdot,s} \log p_{\cdot,s} - \sum_{r \geq 2} p_{r,\cdot} \log p_{r,\cdot} \\
    &=& -\sum_{1 \leq s \leq k} p_{\cdot,s} \log p_{\cdot,s} + p_{\cdot,k} \log p_{\cdot,k} - \sum_{r\geq2} p_{r,\cdot} \log p_{r,\cdot} + p_{1,\cdot} \log p_{1,\cdot} \\
    &=& H(X) + H(Y) + \frac{1}{k} \log \frac{1}{k} + \frac{k-1}{k} \log \frac{k-1}{k} =\max H(X,Y)-K_k
\end{eqnarray*}

\noindent with  $K_k>0$. Since $K_k$ tends to zero when $k \rightarrow \infty$, the minimum of $H(X,Y)$ can be near to the maximum how much as we want. But the equality between the maximum and the minimum cannot be reached, because when $k \rightarrow \infty$ one of the marginal distributions disappears.

\end{document}